\documentclass[10pt]{article}

\usepackage[utf8]{inputenc}
\usepackage{amsfonts}
\usepackage[english]{babel}
\usepackage{amsmath}
\usepackage{amssymb}
\usepackage{soul, color, xcolor}

\usepackage{listings}

\definecolor{mygreen}{rgb}{0,0.6,0}
\definecolor{mygray}{rgb}{0.5,0.5,0.5}
\definecolor{mymauve}{rgb}{0.58,0,0.82}
\definecolor{myorange}{HTML}{d15115}
\definecolor{teal}{HTML}{23ad90}

\usepackage{amsthm}
\usepackage{verbatim}

\usepackage[makeroom]{cancel}
\usepackage{tikz}
\usepackage{mathtools}
\DeclareMathAlphabet{\pazocal}{OMS}{zplm}{m}{n}
\usepackage [autostyle, english = american]{csquotes} 
\MakeOuterQuote{"}
\usepackage{xargs}
\usepackage{booktabs}
\usepackage{multirow}
\usepackage{pdflscape}

\usepackage[numbers]{natbib}
\usepackage{enumitem}
\usepackage[backref=page]{hyperref}
\hypersetup{					 
   colorlinks,
   linkcolor={red},
   citecolor={blue},
   urlcolor={blue!80!black}
}

\newtheorem{theore}{Theorem}[section]

\newtheorem{cor}[theore]{Corollary}

\newtheorem{lem}[theore]{Lemma}
\newenvironment{lemma}
{%
\pushQED{\qed}\begin{lem}}
{\popQED\end{lem}}

\newtheorem{pro}[theore]{Proposition}
\newenvironment{proposition}
{%
\pushQED{\qed}\begin{pro}}
{\popQED\end{pro}}

\theoremstyle{definition}
\newtheorem{defi}{Definition}[section]

\newtheorem{exa}{Example}[section]
\newenvironment{example}
{%
\pushQED{\qed}\begin{exa}}
{\popQED\end{exa}}

\theoremstyle{remark}
\newtheorem{rem}{Remark}
\newenvironment{remark}
{%
\pushQED{\qed}\begin{rem}}
{\popQED\end{rem}}

\usetikzlibrary{positioning,arrows,petri,calc,decorations.markings,arrows.meta}
\tikzset{
place/.style={
circle,
thick,
minimum size=6mm,
draw
},
transitionV/.style={
rectangle,
thick,
fill=black,
minimum height=6mm,
inner xsep=1pt
}
}
\definecolor{myblue}{RGB}{0, 101, 202} 
\definecolor{myred}{RGB}{197, 14, 31} 

\renewcommand*{\backref}[1]{}
\renewcommand*{\backrefalt}[4]{%
    \ifcase #1 (Not cited.)%
    \or        (Cited on page~#2.)%
    \else      (Cited on pages~#2.)%
    \fi}

\title{
    The Hadamard product, its residual, and its dual residual in the dioid of counters: algorithms and implementation in \CCtitle
}
\author{Davide Zorzenon, Germano Schafaschek,\\Dominik Tirpák, Soraia Moradi,\\Laurent Hardouin, and J\"{o}rg Raisch%
\thanks{D. Zorzenon, G. Schafaschek, D. Tirpák, S. Moradi, and J. Raisch are with Technische Universit\"at Berlin, Control Systems Group, Einsteinufer 17, D-10587 Berlin, Germany (e-mail: [zorzenon,schafaschek,moradi,raisch]@control.tu-berlin.de, dominik.tirpak@campus.tu-berlin.de).
J. Raisch is also with Science of Intelligence, Research Cluster of Excellence, Berlin, Germany.
L. Hardouin is with Laboratoire Angevin de Recherche en Ing\'{e}nierie des Syst\`{e}mes, Polytech Angers, Universit\'{e} d’Angers, France (e-mail: laurent.hardouin@univ-angers.fr).}
}

\begin{document}

\newcommand{\minmaxgd}{\mathcal{M}_{in}^{ax}[\![\gamma,\delta]\!]}
\newcommand{\D}{\pazocal{D}}
\newcommand{\CCC}{\pazocal{C}}
\newcommand{\X}{\pazocal{X}}
\newcommand{\N}{\mathbb{N}}
\newcommand{\No}{\mathbb{N}_0}
\newcommand{\Z}{\mathbb{Z}}
\newcommand{\Zmax}{{\Z}_{\normalfont\fontsize{7pt}{11pt}\selectfont\mbox{max}}}
\newcommand{\Zmin}{{\Z}_{\normalfont\fontsize{7pt}{11pt}\selectfont\mbox{min}}}
\newcommand{\Zbar}{\overline{\Z}}
\newcommand{\Zminbar}{\overline{\Z}_{\normalfont\fontsize{7pt}{11pt}\selectfont\mbox{min}}}
\newcommand{\Zmaxbar}{\overline{\Z}_{\normalfont\fontsize{7pt}{11pt}\selectfont\mbox{max}}}
\newcommand{\R}{\mathbb{R}}
\newcommand{\Rmax}{{\R}_{\normalfont\fontsize{7pt}{11pt}\selectfont\mbox{max}}}
\newcommand{\Rmin}{{\R}_{\normalfont\fontsize{7pt}{11pt}\selectfont\mbox{min}}}
\newcommand{\Rbar}{\overline{\R}}
\newcommand{\floor}[1]{\left\lfloor#1\right\rfloor}
\newcommand{\ceil}[1]{\left\lceil#1\right\rceil}
\newcommand\myOverwrite[2]{\!\makebox[0cm][l]{#1}#2\ \!}
\newcommand{\ldiv}{\textup{\hspace{0.85mm}\myOverwrite{$\circ$}{$\backslash$}\hspace{-0.5mm}}}
\newcommand{\rdiv}{\textup{\hspace{0.85mm}\myOverwrite{$\circ$}{$/$}\hspace{-0.5mm}}}
\newcommand{\CC}{C\nolinebreak\hspace{-.05em}\raisebox{.4ex}{\tiny\bf +}\nolinebreak\hspace{-.10em}\raisebox{.4ex}{\tiny\bf +}}
\newcommand{\CCtitle}{C\nolinebreak\hspace{-.05em}\raisebox{.4ex}{\small\bf +}\nolinebreak\hspace{-.10em}\raisebox{.4ex}{\small\bf +}}

\date{}
\maketitle

\begin{abstract}
This report presents the algorithms for computing the Hadamard product, its residual, and its dual residual between formal power series in the dioid of counters (which is isomorphic to $\minmaxgd$, see~\cite{cohen1993two}).
The algorithms have been implemented in the \CC~toolbox ETVO ((Event|Time)-Variant Operators)~\cite{cottenceau2020c++}, which, in turn, is based on the toolbox MinMaxgd~\cite{hardouin2013minmaxgd}.
In this report, after the preliminaries (Section~\ref{se:1}) and the definition of the Hadamard product and its residuals (Section~\ref{se:2}), we present the algorithms and the proofs of their correctness (Section~\ref{se:3}).
Sections~\ref{se:1} and~\ref{se:2} are taken from~\cite{zorzenon2022implementation}, as well as the introduction of Section~\ref{se:3}.
The last part of the report (Section~\ref{se:4}) constitutes a user guide for the \CC~implementation of the algorithms in ETVO, which is available at~\cite{Hadamardproductcounters}.

\vspace{10pt}
\noindent
The algorithms described in this report were found by S. Moradi before 2019, under the supervision of L. Hardouin and J. Raisch.
The proofs of their correctness were derived in the current form by D. Zorzenon in 2022 with the help of G. Schafaschek and D. Tirpák, on the basis of the previous work of S. Moradi.
The algorithms were implemented in \CC~by D. Zorzenon in 2020, and the implementation was improved by D. Tirpák in 2021 under the supervision of G. Schafaschek and D. Zorzenon.
\end{abstract}

\section{Preliminaries}\label{se:1}

\subsection{Dioid theory}

A dioid (or idempotent semiring) $(\D,\oplus,\otimes)$ is a set $\D$ equipped with two binary operations, $\oplus$ and $\otimes$, called respectively addition and multiplication, having the following properties.
Addition is commutative, associative, idempotent (i.e., $a\oplus a = a$ $\forall  a\in\D$), and admits neutral (or zero) element $\varepsilon$; multiplication is associative, distributes over addition, admits neutral (or unit) element $e$, and $\varepsilon$ is absorbing for multiplication (i.e., $a\otimes\varepsilon=\varepsilon\otimes a = \varepsilon$ $\forall a\in\D$).
As in standard algebra, the multiplication symbol "$\otimes$" will be often omitted.
Operation $\oplus$ induces an order relation $\preceq$, defined by $a\preceq b\ \Leftrightarrow \ a\oplus b = b$.

A dioid is complete if it is closed for infinite sums and if multiplication distributes over infinite sums, i.e., $a\otimes\left(\bigoplus_{x\in\X}x\right)= \left(\bigoplus_{x\in\X}a\otimes x\right)$, and $\left(\bigoplus_{x\in\X}x\right)\otimes a= \left(\bigoplus_{x\in\X}x\otimes a\right)$ for all $a\in\D$, $\X\subseteq \D$.
Let $(\D,\oplus,\otimes)$ be a complete dioid.
Its top element is defined by $\top=\bigoplus_{x\in\D}x$.
The greatest lower bound $\wedge$ is defined, for all $a,b\in\D$, by $a\wedge b = \bigoplus_{\D_{ab}} x$, where $\D_{ab} = \{x\in\D\ |\ x\preceq a,\ x\preceq b\}$.
Operation $\wedge$ is commutative, associative, idempotent, and admits $\top$ as neutral element.
Moreover, in a complete dioid $(\D,\oplus,\otimes)$, the Kleene star operator $ ^*$ applied to $a\in\D$ yields $a^*=\bigoplus_{k\in\Z,k\geq 0} a^k$, where $a^0 = e$, and $a^{k+1}=a\otimes a^k$ for all $k\geq 0$.

\begin{remark}\label{rem:oplus_property}
The following equivalence holds: $\forall a,b,c\in\D$, 
$
    a\succeq b \mbox{ and }a\succeq c \ \Leftrightarrow \ a\succeq b\oplus c.
$
Indeed, ($\Leftarrow$) comes from $b\oplus c\succeq b$, $b\oplus c\succeq c$.
($\Rightarrow$) comes from: $a\succeq b\Leftrightarrow a\oplus b = a$, $a\succeq c\Leftrightarrow a\oplus c = a$; therefore, $a\oplus (b\oplus c) = (a\oplus b ) \oplus c = a \oplus c = a$, which is equivalent to $a \succeq b\oplus c$.
Analogously, it is possible to show that: $
    a\preceq b \mbox{ and }a \preceq c \ \Leftrightarrow \ a\preceq b\wedge c.
$
\end{remark}

As in standard algebra, operations $\oplus$ and $\otimes$ can be extended to matrices as follows: for all $A,B\in\D^{m\times n}$ and $C\in\D^{n\times p}$, $A\oplus B\in\D^{m\times n}$ and $A\otimes C\in\D^{m\times p}$ are defined by
\[
    (A\oplus B)_{ij} = A_{ij}\oplus B_{ij},\quad (A\otimes C)_{ij} = \bigoplus_{k=1}^{n} A_{ik}\otimes C_{kj}.
\]
If $(\D,\oplus,\otimes)$ is a complete dioid, then $(\D^{n\times n},\oplus,\otimes)$, where $\oplus$ and $\otimes$ are extended as above, is also a complete dioid.
Its zero (resp. top) element is the $n\times n$-matrix with all entries equal to $\varepsilon$ (resp. $\top$), and its unit element is the $n\times n$-matrix with $e$'s on the main diagonal and $\varepsilon$'s elsewhere.

\begin{example}
An example of complete dioid is the set $\Zbar = \Z\cup\{-\infty,+\infty\}$, with the standard minimum operation as $\oplus$ and standard addition as $\otimes$.
With this notation, the complete dioid $\Zminbar\coloneqq (\Zbar,\oplus,\otimes)$ is called the \emph{min-plus algebra}.
In $\Zminbar$, $\varepsilon=+\infty$, $e=0$, $\top=-\infty$, $\wedge$ corresponds to the standard maximum operation, and $\preceq$ corresponds to the standard $\geq$; this means that the order $\preceq$ is reversed with respect to the conventional one (e.g., $5\preceq 2$).
The dual dioid of $\Zminbar$, denoted $\Zmaxbar$, corresponds to the set $\Zbar$ with the standard maximum operation as $\oplus$ and standard addition as $\otimes$; observe that the order in $\Zmaxbar$ coincides with the standard one.
Due to the absorbing property of $\varepsilon$, the result of $-\infty\otimes +\infty=+\infty\otimes -\infty$ is different in $\Zminbar$ and $\Zmaxbar$.
\end{example}

A mapping $\Pi:\D\rightarrow \CCC$, where $(\D,\oplus,\otimes)$ and $(\CCC,\oplus,\otimes)$ are two dioids, is \emph{isotone} or \emph{non-decreasing} (resp. \emph{antitone} or \emph{non-increasing}) if $\forall a,b\in\D$, $a\preceq b\Rightarrow \Pi(a)\preceq \Pi(b)$ (resp. $\Pi(a)\succeq \Pi(b)$).

\begin{example}\label{ex:counters}
Another example of a complete dioid is the algebra of counters.
Let $s:\Zmaxbar\rightarrow \Zminbar$, $t\mapsto s(t)$, be an antitone mapping such that\footnote{The importance of the end-point conditions on $s$ is explained in~\cite[Chapter~5]{baccelli1992synchronization}.} $s(-\infty)=-\infty$ and $s(+\infty)=+\infty$.
(Note that, due to the reversed order of $\Zminbar$, such mappings are non-decreasing in the standard sense.) 
This kind of mappings can be used to represent the cumulative number $s(t)$ of firings of a transition in a TEG up to and including time $t$.
The \emph{$\delta$-transform} of $s$, called \emph{counter}, is the non-increasing formal power series in $\delta$ with coefficients $s(t)$ in $\Zminbar$ and exponents $t$ in $\Zmaxbar$, defined by 
\[
    s = \bigoplus_{t\in\Zbar} s(t)\delta^t.
\]
As no ambiguity will occur, we indicate both the mapping and its $\delta$-transform by the same symbol.
Since counters are non-increasing and such that $s(-\infty)=-\infty$, $s(+\infty)=+\infty$, we can represent them compactly by omitting terms $-\infty\delta^{-\infty}$, $+\infty\delta^{+\infty}$, and all terms $s(t)\delta^t$ such that $s(t)=s(t+1)$.
For instance, 
\[
    -\infty\delta^{-\infty}\oplus\bigoplus_{-\infty<t\leq 1}-2\delta^t \oplus \bigoplus_{2\leq t\leq 5} 3\delta^t \oplus \bigoplus_{t\geq 6}+\infty\delta^t
\]
will be simply denoted $-2\delta^1\oplus 3\delta^5$.
Thus, we will often avoid mentioning the coefficients of counters for $\delta$-exponents equal to $\pm\infty$, implicitly assuming that the end-point conditions hold.
The set of counters, denoted $\Sigma$, equipped with operations $\oplus$ and $\otimes$ defined, $\forall t\in\Z$, by
\[
    (s\oplus s')(t) = s(t)\oplus s'(t),\quad (s\otimes s')(t) = \bigoplus_{\tau\in\Z}s(\tau)\otimes s'(t-\tau)
\]
is a complete dioid, where the zero, unit, and top element are, respectively, $s_\varepsilon = \bigoplus_{t\in\Z}+\infty\delta^t$, $s_e = e\delta^0$, and $s_\top = \bigoplus_{t\in\Z} -\infty\delta^t$.
Note that, given two counters $s,s'\in\Sigma$, $s\preceq s' \Leftrightarrow s(t)\preceq s'(t)$ for all $t\in\Z$, and their greatest lower bound is given, $\forall t\in\Z$, by $(s\wedge s')(t) = s(t)\wedge s'(t)$.
\end{example}

For algorithmic reasons, it is convenient to distinguish three increasingly larger classes of counters: monomials, of the form $n\delta^t$, polynomials, of the form $\bigoplus_{i=1}^m n_i\delta^{t_i}$ with $m>0$ (and the convention that $n_{i+1}>n_i$, $t_{i+1}>t_i$), and ultimately periodic series.
Series of the third kind are all those that can be written as $s = p\oplus qr^*$, where $p=\bigoplus_{i=1}^m n_i\delta^{t_i}$ is the transient part of $s$ and $q=\bigoplus_{i=1}^l N_i \delta^{T_i}$ is the periodic pattern of $s$, whose periodicity is described by the monomial $r=\nu\delta^\tau$.
For the sake of brevity, we will often write "periodic series" in place of "ultimately periodic series".
In general, the following proposition holds for periodic series.

\begin{proposition}\label{pr:periodic_series}
A formal power series $s$ (which is not necessarily a counter) is said to be periodic if there exist $T_1\in\Z$ (the beginning of the periodic regime), $\nu\in\Z$ (the number of units that $s$ gains after each period), and $\tau\geq 0$ (the period) such that, for all $t\geq T_1$, $k\geq 0$,
\[
    s(t+k\tau) = k\nu + s(t).
\]
If $s$ is a periodic counter, then it can be written as $s=p\oplus qr^*$, where $T_1$ is the $\delta$-exponent of the first monomial of $q$, and $r=\nu\delta^\tau$.
\end{proposition}

When $s$ represents the cumulative firings of a transition in a TEG, the sequence of firings specified by $q$ repeats every $\tau$ time units and after $\nu$ firings of the corresponding transition. 
The ratio $\nu/\tau$ is called \emph{throughput}, and it represents the average number of firings of the transition per unit of time during the periodic regime.
The representation of ultimately periodic series in the form $p\oplus qr^*$ is not unique; however, every ultimately periodic series admits a unique canonical form, in which $m$ (i.e., the number of monomials in the transient part $p$) is minimal.
For example, the canonical form of series $0\delta^1\oplus1\delta^3\oplus(2\delta^6\oplus3\delta^8)(2\delta^5)^*$, graphically represented in Figure~\ref{fig:canonical_form}, is $(0\delta^1\oplus1\delta^3)(2\delta^5)^*$.

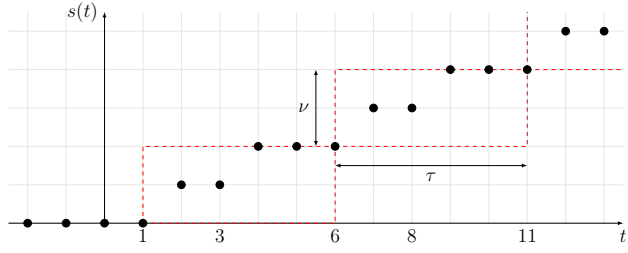
\begin{figure}
    \centering
    \resizebox{.7\linewidth}{!}{
        \begin{tikzpicture}
\Large
\draw[help lines,xstep=1,ystep=1,gray!20] (-2.5,0) grid (13.5,5.5);
\draw[-latex] (-2.5,0)--(13.5,0);
\draw[-latex] (0,0)--(0,5.5);
\node[anchor=north] at (13.5,0) {$t$};
\node[anchor=east] at (0,5.5) {$s(t)$};

\draw [red,dashed] (1,0) -- (1,2) -- (6,2) -- (6,0) -- cycle;
\draw [red,dashed] (6,2) -- (6,4) -- (11,4) -- (11,2) -- cycle;
\draw [red,dashed] (11,4) -- (11,5.5);
\draw [red,dashed] (11,4) -- (13.5,4);

\filldraw (-2,0) circle (.1);
\filldraw (-1,0) circle (.1);
\filldraw (0,0) circle (.1);
\filldraw (1,0) circle (.1);
\filldraw (2,1) circle (.1);
\filldraw (3,1) circle (.1);
\filldraw (4,2) circle (.1);
\filldraw (5,2) circle (.1);
\filldraw (6,2) circle (.1);
\filldraw (7,3) circle (.1);
\filldraw (8,3) circle (.1);
\filldraw (9,4) circle (.1);
\filldraw (10,4) circle (.1);
\filldraw (11,4) circle (.1);
\filldraw (12,5) circle (.1);
\filldraw (13,5) circle (.1);

\foreach \x in {1,3,6,8,11}{
\node[anchor=north] at (\x,0) {$\x$};
}

\draw [latex-latex] (6,1.5) -- node[below] {$\tau$} (11,1.5);
\draw [latex-latex] (5.5,2) -- node[left] {$\nu$} (5.5,4);


\end{tikzpicture}
    }
    \caption{Series $(0\delta^1\oplus1\delta^3)(2\delta^5)^*$.}
    \label{fig:canonical_form}
\end{figure}

\subsection{Residuation theory}\label{su:residuation_theory}

To solve control problems, it is often necessary to compute the inverse of a certain mapping.
When the mapping is not invertible, sometimes it is possible to find the best under- and over-approximation of its inverse, called respectively its \emph{residual} and \emph{dual residual}.

Let $(\D,\oplus,\otimes)$ and $(\CCC,\oplus,\otimes)$ be two complete dioids, and $\Pi:\D\rightarrow\CCC$ an isotone mapping.
The mapping $\Pi$ is \emph{residuated} (resp. \emph{dually residuated}) if, for all $y\in\CCC$, set $\{x\in\D\ |\ f(x)\preceq y \}$ admits maximum (resp. $\{x\in\D\ |\ f(x)\succeq y \}$ admits minimum).
In this case, the mapping $f^\sharp:\CCC\rightarrow\D$, $y\mapsto \bigoplus\{x\in\D\ |\ f(x)\preceq y\}$ (resp. $f^\flat:\CCC\rightarrow\D$, $y\mapsto \bigwedge\{x\in\D\ |\ f(x)\succeq y\}$) is called the residual (resp. dual residual) of $f$.

\begin{remark}\label{rem:residual_property}
If $\Pi$ is residuated (resp. dually residuated), its residual (resp. dual residual) is isotone,
\[
   \Pi \circ \Pi^\sharp \preceq \mbox{Id}_\CCC, \quad \mbox{and}\quad \Pi^\sharp\circ\Pi \succeq \mbox{Id}_\D 
\] 
\[
   (\mbox{resp.} \ \Pi \circ \Pi^\flat \succeq \mbox{Id}_\CCC,\quad\mbox{and}\quad\Pi^\flat\circ\Pi\preceq\mbox{Id}_\D),
\]
where $\mbox{Id}_\D$ and $\mbox{Id}_\CCC$ indicate the identity mappings in $\D$ and $\CCC$, respectively.
From these observations we can derive the following properties:
\[
   \Pi(x) \preceq y \ \Leftrightarrow \ x \preceq \Pi^\sharp(y) ,
\]
\[
   \Pi(x) \succeq y \ \Leftrightarrow \ x \succeq \Pi^\flat(y) .
\]
We prove only the former, since the proof of the latter is analogous.

"$\Leftarrow$": since $\Pi$ is isotone, $x\preceq \Pi^\sharp(y)$ implies $\Pi(x) \preceq \Pi(\Pi^\sharp(y))$, and since $\Pi\circ \Pi^\sharp \preceq \mbox{Id}_\CCC$, $\Pi(\Pi^\sharp(y)) \preceq y$.

"$\Rightarrow$": since $\Pi^\sharp$ is isotone, $\Pi(x) \preceq y$ implies $\Pi^\sharp(\Pi(x)) \preceq \Pi^\sharp(y)$, and since $\mbox{Id}_\D \preceq \Pi^\sharp\circ\Pi$, $x \preceq \Pi^\sharp(\Pi(x))$.
\end{remark}

\section{The Hadamard product of counters: definition and residuals}\label{se:2}

In this section, we define the Hadamard product and its residuals; these operations are useful for solving optimal-control problems for some interesting classes of discrete event systems, as will be discussed in the next section.

The Hadamard product of two counters $s_1,s_2\in\Sigma$, denoted by $s_1\odot s_2$, is defined by
\[
    (s_1 \odot s_2) (t) = s_1(t) \otimes s_2(t) \quad \forall t\in\Z.
\]
In standard algebra, it corresponds to the element-wise addition of the coefficients of the corresponding series.
We recall from~\cite{4605920} that $\odot$ is commutative and distributes over finite $\wedge$.

\begin{rem}\label{rem:difference_and_residuals}
Note that, given two counters $x,a\in\Sigma$, the series $\Pi_a(x)=a\odot x$ is always a counter.
On the other hand, given $y,a\in\Sigma$, the same is not always true for the series $\bar{x}$ defined by $\bar{x}(t) = y(t) - a(t)$ $\forall t\in\Z$, $\bar{x}(-\infty)=-\infty$, $\bar{x}(+\infty) = +\infty$. 
As the following discussion will reveal, $\Pi_a$ is both residuated and, under certain conditions, dually residuated.
Hence, the greatest counter less than or equal to $\bar{x}$ (in the sense of the order in $\Sigma$) is given by the residual of $\Pi_a$, $\Pi_a^\sharp(y) = y\odot^\sharp a$, and the least counter greater than or equal to $\bar{x}$, when defined, is given by the dual residual of $\Pi_a$, $\Pi_a^\flat(y) = y\odot^\flat a$.
The difference between series $\bar{x}$, $y\odot^\sharp a$, and $y\odot^\flat a$ is shown through an example in Figure~\ref{fig:difference_and_residuals}.
Let us now formally characterize the operations $\odot^\sharp$ and $\odot^\flat$. \hfill $\Diamond$
\end{rem}

\tikzset{mycircle/.style={draw, black, circle, minimum size=.1}}
\newsavebox{\nodemycircle}
\sbox\nodemycircle{\tikz{\node[mycircle]{\!};}}
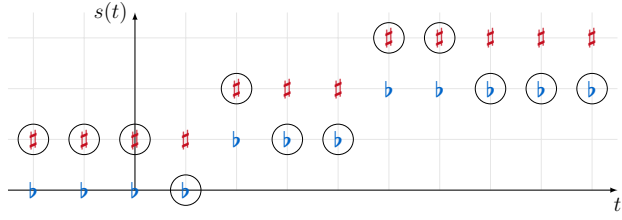
\begin{figure}
    \centering
    \resizebox{.7\linewidth}{!}{
        \begin{tikzpicture}
\large
\draw[help lines,xstep=1,ystep=1,gray!20] (-2.5,0) grid (9.5,3.5);
\draw[-latex] (-2.5,0)--(9.5,0);
\draw[-latex] (0,0)--(0,3.5);
\node[anchor=north] at (9.5,0) {$t$};
\node[anchor=east] at (0,3.5) {$s(t)$};


\draw (-2,1) circle (.3);
\draw (-1,1) circle (.3);
\draw (0,1) circle (.3);
\draw (1,0) circle (.3);
\draw (2,2) circle (.3);
\draw (3,1) circle (.3);
\draw (4,1) circle (.3);
\draw (5,3) circle (.3);
\draw (6,3) circle (.3);
\draw (7,2) circle (.3);
\draw (8,2) circle (.3);
\draw (9,2) circle (.3);

\node [anchor=center,myred] at (-2,1) {$\boldsymbol{\sharp}$};
\node [anchor=center,myred] at (-1,1) {$\boldsymbol{\sharp}$};
\node [anchor=center,myred] at (0,1) {$\boldsymbol{\sharp}$};
\node [anchor=center,myred] at (1,1) {$\boldsymbol{\sharp}$};
\node [anchor=center,myred] at (2,2) {$\boldsymbol{\sharp}$};
\node [anchor=center,myred] at (3,2) {$\boldsymbol{\sharp}$};
\node [anchor=center,myred] at (4,2) {$\boldsymbol{\sharp}$};
\node [anchor=center,myred] at (5,3) {$\boldsymbol{\sharp}$};
\node [anchor=center,myred] at (6,3) {$\boldsymbol{\sharp}$};
\node [anchor=center,myred] at (7,3) {$\boldsymbol{\sharp}$};
\node [anchor=center,myred] at (8,3) {$\boldsymbol{\sharp}$};
\node [anchor=center,myred] at (9,3) {$\boldsymbol{\sharp}$};

\node [anchor=center,myblue] at (-2,0) {$\boldsymbol{\flat}$};
\node [anchor=center,myblue] at (-1,0) {$\boldsymbol{\flat}$};
\node [anchor=center,myblue] at (0,0) {$\boldsymbol{\flat}$};
\node [anchor=center,myblue] at (1,0) {$\boldsymbol{\flat}$};
\node [anchor=center,myblue] at (2,1) {$\boldsymbol{\flat}$};
\node [anchor=center,myblue] at (3,1) {$\boldsymbol{\flat}$};
\node [anchor=center,myblue] at (4,1) {$\boldsymbol{\flat}$};
\node [anchor=center,myblue] at (5,2) {$\boldsymbol{\flat}$};
\node [anchor=center,myblue] at (6,2) {$\boldsymbol{\flat}$};
\node [anchor=center,myblue] at (7,2) {$\boldsymbol{\flat}$};
\node [anchor=center,myblue] at (8,2) {$\boldsymbol{\flat}$};
\node [anchor=center,myblue] at (9,2) {$\boldsymbol{\flat}$};

\end{tikzpicture}
    }
    \caption{Series $\bar{x}$, $y\odot^\sharp a$, and $y\odot^\flat a$, indicated, respectively, by \usebox{\nodemycircle}'s, \textcolor{myred}{$\boldsymbol{\sharp}$}'s, and \textcolor{myblue}{$\boldsymbol{\flat}$}'s, when $y = 1\delta^{1}\oplus 3\delta^4\oplus 5\delta^{+\infty}$ and $a=0\delta^0\oplus 1\delta^2 \oplus 2\delta^6\oplus 3\delta^{\infty}$. Note that $\bar{x}$ is not a counter.}
    \label{fig:difference_and_residuals}
\end{figure}

The following result, proven in~\cite{4605920}, shows that the Hadamard product is residuated.

\begin{proposition}
The mapping $\Pi_a:\Sigma\rightarrow\Sigma$, $x\mapsto a\odot x$ is residuated for any $a\in\Sigma$.
Its residual is denoted by $\Pi_a^\sharp(y) = y\odot^\sharp a$, and corresponds to the greatest counter $x\in\Sigma$ that satisfies $a\odot x\preceq y$.
\end{proposition}

In general, however, the mapping $\Pi_a$ is not dually residuated.
Indeed, if for a certain $t\in\Z$ $a(t)=+\infty$ and $y(t)\neq +\infty$, the least solution $x$ of $a\odot x\succeq y$ is not defined, as inequality $a(t)\otimes x(t) \succeq y(t)$ (in standard algebra, $+\infty + x(t) \leq y(t)$) does not admit solutions.
Another situation in which $\Pi_a$ is not dually residuated is when there exists $t\in\Z$ such that $a(t) = -\infty$ and $y(t) \neq +\infty$; in this case, inequality $a(t)\otimes x(t) \succeq y(t)$ (in standard algebra, $-\infty + x(t) \leq y(t)$) admits infinitely many solutions, but the infimum of the solution set, $+\infty$, does not belong to it.
The following proposition, proven in~\cite{zorzenon2022implementation}, shows that these two are the only cases in which the dual residual of $\Pi_a$ is not defined.

\begin{proposition}\label{pr:dual_residual}
For $a\in\Sigma$, let $\D_a = \{x\in\Sigma\ |\ x = s_\varepsilon \text{ if }\exists t\in\Z \text{ with }a(t) = -\infty\}$, and $\CCC_a = \{y\in\Sigma\ |\ y(t) = +\infty \ \forall t\in\Z \text{ such that }a(t) \in\{-\infty,+\infty\}\}$.
The mapping $\Pi_a:\D_a\rightarrow\CCC_a$, $x\mapsto a\odot x$ is dually residuated for any $a\in\Sigma$.
Its dual residual is denoted by $\Pi_a^\flat(y) = y\odot^\flat a$, and corresponds to the least counter $x\in\Sigma$ that satisfies $a\odot x\succeq y$.
\end{proposition}

Note that, for $y\odot^\flat a$ to be defined for any $y\in\Sigma$, it suffices that $a(t) \neq \pm \infty$ for all $t\in\Z$.
This condition is not restrictive for application purposes, as $a(t)$ will typically denote the (finite) accumulated number of firings of a transition up to and including time $t$.
Hence, Propositions~\ref{pr:Hadamard_res} and~\ref{pr:dual_residual} guarantee the existence of the residual and dual residual of the Hadamard product for any case of practical interest.
In the next subsection, we see how to compute the results of these operations.

\section{Algorithms for the Hadamard product and its residuals}\label{se:3}

In the following subsections, the algorithms for computing the Hadamard product, its residual, and its dual residual are described, and their correctness proven.
In order to implement operations on non-increasing formal power series, it is convenient to consider separately monomials, polynomials, and ultimately periodic series.
We will always use the following notation:
\[
    r=\nu\delta^{\tau},\quad r'=\nu'\delta^{\tau'},
\]
\[
   p=\bigoplus_{i=1}^m n_i\delta^{t_i},\quad  p' = \bigoplus_{j=1}^{m'} n_j'\delta^{t_j'},
\]
\[
    q=\bigoplus_{i=1}^l N_i\delta^{T_i},\quad q'=\bigoplus_{j=1}^{l'} N_j'\delta^{T_j'}, 
\]
\[
    s=p\oplus qr^*,\quad s'=p'\oplus q'r'^*,
\]
where $s,s'$ are in canonical form, and the monomials composing polynomials $p,p',q,q'$ are such that their coefficients and $\delta$-exponents are written in increasing order (with respect to the order in standard algebra).
Before going into the details of the different algorithms for monomials, polynomials, and ultimately periodic series, it is useful to summarize them.
The rules for computing these operations are reported in Table~\ref{tab:procedures}, whose interpretation is explained in the following.
The formulas show the rules to compute $\odot$, $\odot^\sharp$, and $\odot^\flat$ between monomials $r$, $r'$, polynomials $p$, $p'$, and ultimately periodic series $s$, $s'$.
Column "Convention $+\infty-\infty$" explains how to interpret the standard additions and subtractions contained in column "Monomials", when $\nu$ or $\nu'$ are $+\infty$ and $-\infty$.

The computation of the Hadamard product, its residual, and its dual residual on monomials and polynomials is straightforward and, for polynomials, the result can be obtained in time complexity $\pazocal{O}(mm')$; the situation is less trivial when considering ultimately periodic series.
An important observation is that applying the Hadamard product and its residuals on two ultimately periodic series $s$ and $s'$ yields another ultimately periodic series $s''$, with throughput $\nu''/\tau''$ and periodic behavior starting at the latest at time $t_{\textup{p}}''$ (the values of $\nu''$, $\tau''$, and $t_{\textup{p}}''$ being reported in the table). 
Consequently, to compute the result of operation $\circ\in\{\odot,\odot^\sharp,\odot^\flat\}$ between $s$ and $s'$, we can adopt the following procedure: obtain the polynomials, say $\tilde{p}$ and $\tilde{p}'$, composed by the first terms of series $s$ and $s'$ up to and including time $t_{\textup{p}}''+\tau''-1$; compute $\tilde{p}''=\tilde{p}\circ\tilde{p}'$; define polynomials $p''$ and $q''$ such that elements of $\tilde{p}''$ with a $\delta$-exponent less than $t_{\textup{p}}''$ belong to $p''$, and those with a $\delta$-exponent between $t_{\textup{p}}''$ and $t_{\textup{p}}'' + \tau'' -1$ belong to $q''$; the result of $s\circ s'$ is then $s'' = p''\oplus q'' (\nu''\delta^{\tau''})^*$.
Note that values of $t_\textup{p}''$ reported in the table are only upper bounds of the starting time of the periodic pattern of $s''$; a formula for the exact starting time is indeed not necessary for computing $s''$.
The only inconvenience is that series $p''\oplus q''(\nu\delta^{\tau''})^*$ obtained in this way may be not in canonical form, resulting in a transient part longer than necessary; nevertheless, rewriting a given series in canonical form is not computationally expensive.

The procedure described above, of complexity $\pazocal{O}(\tilde{m}\tilde{m}')$ where $\tilde{p}=\bigoplus_{i=1}^{\tilde{m}} \tilde{n}_i\delta^{\tilde{t}_i}$ and $\tilde{p}' = \bigoplus_{j=1}^{\tilde{m}'} \tilde{n}_j'\delta^{\tilde{t}_j'}$, can be applied successfully for each operation. 
A simple formula for $t_\textup{p}$ is unknown to the authors for the residual of the Hadamard product, but an upper bound for the beginning of the periodic pattern of $s'' = s\odot^\sharp s'$ can be computed on the basis of the analysis of series $\bar{s}$, defined by $\bar{s}(t) = s(t) - s'(t)$ for all $t\in\Z$; we recall that series $s''$ is then the greatest counter less than or equal to $\bar{s}$ (see Remark~\ref{rem:difference_and_residuals}).
It turns out that we can take $t_\textup{p}''$ as $t_\textup{p}'' = \bar{t}_{\textup{p}}+\kappa\tau''$, with $\bar{t}_\textup{p}=\max(T_1,T_1')$,
\[
    \displaystyle\kappa=1+\max\left(0,\ceil{\frac{\max_{i=\bar{t}_1}^{\bar{t}_\textup{p}-1}\bar{s}(i)-\max_{j=\bar{t}_\textup{p}}^{\bar{t}_\textup{p}+\tau''-1}\bar{s}(j)}{\nu''}}\right),
\] 
and $\bar{t}_1=\min(t_1,t_1')$.

\begin{landscape}
\begin{table}[t]
\centering
\caption{Rules for computing $\odot$, $\odot^\sharp$, and $\odot^\flat$ between monomials $r=\nu\delta^{\tau}$ and $r'=\nu'\delta^{\tau'}$, polynomials $p=\bigoplus_{i=1}^m n_i\delta^{t_i}$ and $p' = \bigoplus_{j=1}^{m'} n_j'\delta^{t_j'}$, and ultimately periodic series $s=p\oplus qr^*$ and $s'=p'\oplus q'r'^*$, where $q=\bigoplus_{i=1}^l N_i\delta^{T_i}$ and $q'=\bigoplus_{j=1}^{l'} N_j'\delta^{T_j'}$.}
\label{tab:procedures}
\begin{tabular}{@{}lllllll@{}}
\toprule
\multirow{2}{*}{} & \multirow{2}{*}{\begin{tabular}[c]{@{}l@{}}Convention\\ $+\infty-\infty$\end{tabular}} & \multirow{2}{*}{Monomials}                                                                                             & \multirow{2}{*}{Polynomials}                                                                                                                                                                                                                                                                        & \multicolumn{3}{c}{Ultimately periodic series}                                                                  \\ \cmidrule(l){5-7} 
                  &                                                                                        &                                                                                                                        &                                                                                                                                                                                                                                                                                                     & $\tau''$                   & $\nu''$                                                  & $t_{\textup{p}}''$            \\ \midrule
$\odot$           & $+\infty$                                                                              & $(\nu+\nu')\delta^{\min(\tau,\tau')}$                                                                                            & $\displaystyle\bigoplus_{i=1}^{m}\bigoplus_{j=1}^{m'}(n_i\delta^{t_i}\odot n_j'\delta^{t_j'})$                                                                                                                                                                                                      & $\mbox{lcm}(\tau,\tau')$ & $\tau''\left(\frac{\nu}{\tau}+\frac{\nu'}{\tau'}\right)$ & $\max(T_1,T_1')$ \\
$\odot^\sharp$    & $-\infty$                                                                              & $\displaystyle\begin{dcases} (\nu-\nu')\delta^\tau & \mbox{if }\tau<\tau',\\ (\nu-\nu')\delta^{+\infty} & \mbox{otherwise}\end{dcases}$ & $\displaystyle\bigwedge_{j=1}^{m'}\bigoplus_{i=1}^{m}(n_i\delta^{t_i}\odot^\sharp n_j'\delta^{t_j'})$                                                                                                                                                                                               & $\mbox{lcm}(\tau,\tau')$ & $\tau''\left(\frac{\nu}{\tau}-\frac{\nu'}{\tau'}\right)$ & see text                \\
$\odot^\flat$     & $+\infty$                                                                              & $\displaystyle\begin{dcases} (\nu-\nu')\delta^\tau & \mbox{if }\tau\leq \tau',\\ \mbox{undefined} & \mbox{otherwise} \end{dcases}$  & \begin{tabular}[c]{@{}l@{}}$\displaystyle\begin{dcases} \bigoplus_{i=1}^{m}\bigwedge_{j=1}^{m'}(n_i\delta^{t_i}\odot^\flat n_j'\delta^{t_j'}) & \mbox{if }t_m\leq t_m',\\ \mbox{undefined} & \mbox{otherwise} \end{dcases}$\\ignoring undefined results on monomials\end{tabular} & $\mbox{lcm}(\tau,\tau')$ & $\tau''\left(\frac{\nu}{\tau}-\frac{\nu'}{\tau'}\right)$ & $\max(T_1,T_1')$ \\ \bottomrule
\end{tabular}
\end{table}
\end{landscape}



\subsection{The Hadamard product}

\subsubsection{The Hadamard product for monomials}

\begin{proposition}[$\odot$ for monomials]
The following rule holds:
\[
    r\odot r' = (\nu+\nu')\delta^{\min(\tau,\tau')}
\]
with the convention that $+\infty-\infty=+\infty$.
\end{proposition}
\begin{proof}
As follows, we only consider $\nu,\nu',\tau,\tau'\in\Z$; the cases in which $\nu,\nu',\tau,\tau'\in\{-\infty,+\infty\}$ have to be considered separately, but can be easily checked by inspection.
The monomials $r$ and $r'$ correspond to the mappings
\[
    r(\tau)=
    \begin{dcases}
        \nu & \mbox{if }t\leq \tau,\\
        +\infty & \mbox{otherwise}
    \end{dcases}
    \quad\mbox{and}\quad
    r'(\tau)=
    \begin{dcases}
        \nu' & \mbox{if }t\leq \tau',\\
        +\infty & \mbox{otherwise}
    \end{dcases}.
\]
Therefore,
\[
    (r\odot r')(\tau) =
    \begin{dcases}
        \nu+\nu' & \mbox{if }t\leq\min(\tau,\tau')\\
        +\infty & \mbox{otherwise}
    \end{dcases},
\]
which can be compactly represented by the monomial $(\nu+\nu')\delta^{\min(\tau,\tau')}$.
\end{proof}

\subsubsection{The Hadamard product for polynomials}

\begin{proposition}[$\odot$ for polynomials]
The following rule holds:
\[
    p\odot p' = \bigoplus_{i=1}^m\bigoplus_{j=1}^{m'} n_i\delta^{t_i}\odot n_j'\delta^{t_j'}.
\]
\end{proposition}
\begin{proof}
The proof comes directly from the distributivity property of $\odot$ over $\oplus$.
\end{proof}

The computation of $p\odot p'$ can be simplified in some cases.
If a pair of monomials with $t_i\leq t_j'$ is encountered, then for the next term we can advance in the computation with $i+1$ and the current $j$ value (without setting it back to 1).
The remaining terms for the current $i$ value would not contribute to the result, as all their exponents would equal $t_i$ and their coefficients would be increasing in the standard sense.
Another possible way to increase the efficiency of this computation is in case $\min(t_i,t_j')=t'_{m'}$.
After calculating the Hadamard product of the monomials with this property, the computation can be terminated, as the remaining terms will again not contribute to the end result.

\begin{example}
The usefulness of the above mentioned simplification rules is evident from the following example:
\[
    \begin{array}{rl}
    (2\delta^2\oplus 3\delta^5 \oplus 7\delta^8) \odot (4\delta^3 \oplus 6\delta^4)\!\!\!\!\! &= (6\delta^2 \oplus \cancel{8\delta^2})\oplus(7\delta^3\oplus 9\delta^4) \oplus (\cancel{11\delta^3}\oplus\cancel{13\delta^4})\\
    &=6\delta^2 \oplus 7\delta^3 \oplus 9\delta^4.
    \end{array}
\]
In this case, the simplification rules essentially halve the amount of operations required.
\end{example}

\subsubsection{The Hadamard product for periodic series}

\begin{proposition}[$\odot$ for periodic series]
Series $s''=s\odot s'$ is periodic, with periodic pattern described by
\[
    \tau'' = \textup{lcm}(\tau,\tau'), \quad \nu'' = \tau''\left(\frac{\nu}{\tau}+\frac{\nu'}{\tau'}\right), \quad t''_{\textup{p}} = \max(T_1,T_1').
\]
\end{proposition}
\begin{proof}
Let $\tau'' = \textup{lcm}(\tau,\tau') = \alpha\tau=\alpha'\tau'$.
Then, since counters $s$ and $s'$ are periodic, for all $t\geq \max(T_1,T_1')=t_{\textup{p}}''$, $k\geq 0$,
\begin{align*}
    s''(t+k\tau'') &= s(t+k\tau'') + s'(t+k\tau'') = s(t+k\alpha\tau)+s'(t+k\alpha'\tau')\\
    &= k \alpha\nu + s(t) + k\alpha'\nu'+s'(t) = k(\alpha\nu+\alpha'\nu')+s''(t).
\end{align*}
Therefore, according to Proposition~\ref{pr:periodic_series}, series $s''$ is periodic with $\tau''$, $\nu''$, and $t''_{\textup{p}}$ as claimed in the statement of the proposition. 
\end{proof}

\begin{example}
\label{ex: hadamard_series}
Let $s = e\delta^1 \oplus (2\delta^3)(3\delta^3)^*$ and $s' = (1\delta^2 \oplus 4\delta^5)(4\delta^6)^* $.
To compute $s'' = s\odot s'$, we first determine the following values: $\tau'' = \textup{lcm}(3,6) = 6,\, \nu'' = 6\left( \frac{3}{3} + \frac{4}{6} \right) = 10,\, t''_{\textup{p}} = \max(3,2) = 3$.
Both series are therefore expanded until $t''_{\textup{p}} + \tau'' - 1 = 8$, resulting in
\[
    \tilde{p} = e\delta^1 \oplus 2\delta^3 \oplus 5\delta^6 \oplus 8\delta^8 \quad\text{and}\quad \tilde{p}' = 1\delta^2 \oplus 4\delta^5 \oplus 5\delta^8.
\]
From this, $\tilde{p}'' = \tilde{p}\odot \tilde{p}' = 1\delta^1 \oplus 3\delta^2 \oplus 6\delta^3 \oplus 9\delta^5 \oplus 10\delta^6 \oplus 13\delta^8$ follows.
Distributing the monomials among $p''$ and $q''$, we find \[ s'' = p''\oplus q''(\nu''\delta^{\tau''})^* = 1\delta^1 \oplus 3\delta^2 \oplus (6\delta^3 \oplus 9\delta^5 \oplus 10\delta^6 \oplus 13\delta^8)(10\delta^6)^*. \]
One may notice that this is equivalent to $s'' = 1\delta^1 \oplus (3\delta^2 \oplus 6\delta^3 \oplus 9\delta^5 \oplus 10\delta^6)(10\delta^6)^*$, which is the canonical form of this series.
\end{example}

The above algorithm, with slight modifications, can also be applied in the special case when one of the operands is represented by a polynomial and not as a periodic series.
Consider the Hadamard product of a periodic series $s = p\oplus q(\nu \delta^{\tau})^*$ and a polynomial $p'\neq s_\varepsilon$.
We need to distinguish between two cases:
\begin{itemize}\itemsep0em
    \item The $\delta$-exponent of the last monomial of $p'$ is $+\infty$, i.e., $t_{m'} = +\infty$.
    In this case, the polynomial $p'$ can also be written as
    \begin{equation*}
    \label{eq: poly_written_as_series}
        p' = \left(\bigoplus_{j=1}^{m'-1} n'_j\delta^{t'_j}\right)\oplus n_{m'}\delta^{1+t_{m'-1}}(0\delta^1)^*,
    \end{equation*}
    for instance, $1\delta^3 \oplus 5\delta^{+\infty} = 1\delta^3 \oplus 5\delta^4 (0\delta^1)^*$.
    This leads to $\tau'' = \textup{lcm}(\tau,1) = \tau$ and $\nu'' = \tau\left( \frac{\nu}{\tau} + \frac{0}{1} \right) = \nu$.
    Furthermore, the periodic behavior of the result starts at the latest at time $t''_{\textup{p}} = \max(T_{1},1+t_{m'-1})$.
    With these values, we can expand the series $s$ until $t''_{\textup{p}} + \tau'' - 1$, and the rest of the algorithm works as described above.
    \item In the opposite case, with $t_{m'}\neq +\infty$, we have $p'(t)=+\infty$ for all $t>t_{m'}$, i.e., the last (implicit) term of polynomial $p'$ is $+\infty\delta^{+\infty}$ (which is omitted in the representation, cf. Example~\ref{ex:counters}). 
    This is therefore true for the result as well (as $+\infty$ is absorbing for $\otimes$).
    In the end, this means that $s\odot p'$ can also be represented by a polynomial.
    Here, since all values of the result will be $+\infty$ after time $t_{m'}$, the series $s$ only has to be expanded up to $t_{m'}$.
    The final result can then be obtained by $s'' = s\odot p'  = \tilde{p}\odot p'$.
\end{itemize}

\begin{example}
Let $s = 1\delta^{2}\oplus (3\delta^{4})(2\delta^{5})^*$ and $p' = 3\delta^{4}\oplus 4\delta^{+\infty}$.
As $p'$ can also be written as $p' = 3\delta^{4}\oplus (4\delta^{5})(0\delta^{1})^*$, the periodic behavior of the result starts at $t''_{\textup{p}} = \max(4,5) = 5$ at the latest.
With $\nu'' = \nu = 2$ and $\tau'' = \tau = 5$, the series $s$ is expanded until $t''_{\textup{p}} + \tau''-1 = 9$, resulting in $\tilde{p} = 1\delta^{2}\oplus 3\delta^{4}\oplus 5\delta^{9}$.
The operation $\tilde{p}\odot p'$ yields $4\delta^2 \oplus 6\delta^4 \oplus 9\delta^9$, so that the final result is equal to
\[ s'' = s\odot p' = 4\delta^{2}\oplus 6\delta^{4}\oplus (9\delta^{9})(2\delta^{5})^*. \]
Now consider the Hadamard product of $s$ with $p' = 3\delta^4 \oplus 4\delta^8$.
The series $s$ now only has to be expanded until $t_{m'} = 8$, and the final result is given by
\[ s'' = s\odot p' = 4\delta^2 \oplus 6\delta^4 \oplus 9\delta^8. \]
\end{example}

\subsection{The residual of the Hadamard product}

\subsubsection{The residual of the Hadamard product for monomials}

\begin{proposition}[$\odot^\sharp$ for monomials]
The following rule holds:
\[
    r\odot^\sharp r' = 
    \begin{dcases}
        (\nu-\nu')\delta^{\tau} & \textup{if }\tau<\tau'\\
        (\nu-\nu')\delta^{+\infty} & \textup{otherwise}
    \end{dcases},
\]
with the convention that $+\infty-\infty=-\infty$.
\end{proposition}
\begin{proof}
As follows, we only consider $\nu,\nu',\tau,\tau'\in\Z$; the cases in which $\nu,\nu',\tau,\tau'\in\{-\infty,+\infty\}$ have to be considered separately, but can be easily checked by inspection.
Counter $r\odot^\sharp r'$ is the greatest solution of $r'\odot x\preceq r$; we distinguish two cases.

\textbf{Case 1: $\tau<\tau'$.}
Equivalently, $r\odot^\sharp r'$ is the counter that solves the following multi-objective optimization problem (written in standard algebra for simplicity; consider always $+\infty-\infty=+\infty$)\footnote{Note that the existence and uniqueness of the solution to this and the following multi-objective optimization problems is guaranteed by the fact that $\Pi_a:\Sigma\rightarrow\Sigma$, $x \mapsto a\odot x$ is residuated for all $a\in\Sigma$.}:
\[
    \begin{array}{rll}
    \min        & x(t) &  \forall t\in\Z, \\
    \mbox{s.t.} & \nu' + x(t) \geq \nu  & \forall t \leq \tau,\\
                & \nu' + x(t) \geq +\infty & \forall \tau < t \leq \tau',\\
                & +\infty + x(t) \geq +\infty & \forall \tau' < t,\\
                & x(t) \leq x(t+1) & \forall t\in\Z.
    \end{array}
\]
It is then clear that the greatest counter that satisfies $r'\odot x\preceq r$ has coefficients
\[
    x(t) = 
    \begin{dcases}
        \nu - \nu' & \mbox{if } t \leq \tau,\\
        +\infty & \mbox{otherwise}
    \end{dcases},
\]
which corresponds to monomial $(\nu-\nu')\delta^{\tau}$.

\textbf{Case 2: $\tau\geq\tau'$.}
Equivalently, $r\odot^\sharp r'$ is the counter that solves the following multi-objective optimization problem (written in standard algebra for simplicity; consider always $+\infty-\infty=+\infty$):
\[
    \begin{array}{rll}
    \min        & x(t) &  \forall t\in\Z, \\
    \mbox{s.t.} & \nu' + x(t) \geq \nu  & \forall t \leq \tau',\\
                & +\infty + x(t) \geq \nu & \forall \tau' < t \leq \tau,\\
                & +\infty + x(t) \geq +\infty & \forall \tau < t,\\
                & x(t) \leq x(t+1) & \forall t\in\Z.
    \end{array}
\]
It is then clear that the greatest counter that satisfies $r'\odot x\preceq r$ has coefficients
\[
    x(t) = \nu - \nu' \quad \mbox{for all }t,
\]
which corresponds to monomial $(\nu-\nu')\delta^{+\infty}$.
\end{proof}

\subsubsection{The residual of the Hadamard product for polynomials}

Before giving the rule for computing $\odot^\sharp$ on polynomials, we need to consider the special case in which the left operand is a polynomial and the right operand a monomial. 

\begin{lemma}[$\odot^\sharp$ between a polynomial and a monomial]\label{le:odot_sharp_poly_mono}
The following rule holds:
\[
    p\odot^\sharp r = \left(\bigoplus_{i=1}^m n_i\delta^{t_i}\right)\odot^\sharp \nu\delta^\tau = \bigoplus_{i=1}^m \left(n_i\delta^{t_i}\odot^\sharp \nu\delta^\tau\right).
\]
\end{lemma}
\begin{proof}
Let $j\in\{0,1,\ldots,m\}$ be the greatest index for which $t_j \leq \tau$, where $t_0 = -\infty$.
The case in which $j = m$ will not be discussed for the sake of brevity, as it needs to be considered separately.
Then,
\begin{equation}\label{eq:poly_sharp_mono}
    \bigoplus_{i=1}^m \left(n_i\delta^{t_i} \odot^\sharp \nu\delta^\tau\right) = \left(\bigoplus_{i = 1}^j (n_i - \nu) \delta^{t_i}\right) \oplus (n_{j+1}-\nu)\delta^{+\infty}.
\end{equation}
We need to prove that $p\odot^\sharp r$ leads to the same expression.

From the definition of residual,
\[
    p\odot^\sharp r = \bigoplus\{x\in\Sigma\ |\ r\odot x \preceq p\}.
\]
Equivalently, $p\odot^\sharp r$ is the counter that solves the following multi-objective optimization problem (written in standard algebra for simplicity; consider always $+\infty-\infty=+\infty$):
\[
    \begin{array}{rll}
    \min        & x(t) &  \forall t\in\Z, \\
    \mbox{s.t.} & \nu + x(t) \geq n_{i}  & \forall i\in\{1,\ldots,j\},\ t_{i-1} < t \leq t_i,\\
                & \nu + x(t) \geq n_{j+1} & \forall t_j < t \leq \tau,\\
                & +\infty + x(t) \geq n_{j+1} & \forall \tau < t \leq t_{j+1},\\
                & +\infty + x(t) \geq n_{i} & \forall i\in\{j+2,\ldots,m\},\ t_{i-1} < t \leq t_{i},\\
                & +\infty + x(t) \geq +\infty & \forall t_m < t,\\
                & x(t) \leq x(t+1) & \forall t\in\Z.
    \end{array}
\]
Clearly, the optimization problem can be rewritten as 
\[
    \begin{array}{rll}
    \min        & x(t) &  \forall t\in\Z, \\
    \mbox{s.t.} & x(t) \geq n_{i} - \nu  & \forall i\in\{1,\ldots,j\},\ t_{i-1} < t \leq t_i,\\
                & x(t) \geq n_{j+1} - \nu & \forall t_j < t \leq \tau,\\
                & x(t) \geq -\infty & \forall \tau < t,\\
                & x(t) \leq x(t+1) & \forall t\in\Z,
    \end{array}
\]
whose solution is
\[
    x(t) = 
    \begin{dcases}
        n_i - \nu & \forall i\in\{1,\ldots,j\},\ t_{i-1} < t \leq t_i,\\
        n_{j+1} - \nu & \forall t_j < t;
    \end{dcases}
\]
written using the $\delta$-transform, the above expression coincides with~\eqref{eq:poly_sharp_mono}.
\end{proof}

\begin{proposition}[$\odot^\sharp$ for polynomials]
The following rule holds:
\[
    p\odot^\sharp p' = \bigwedge_{j=1}^{m'}\bigoplus_{i=1}^{m} n_i\delta^{t_i}\odot^\sharp n_j'\delta^{t_j'}.
\]
\end{proposition}
\begin{proof}
Let $s''=p\odot^\sharp p'$ be the greatest counter satisfying
\[
    p'\odot s'' \preceq p, \quad \text{i.e.,} \quad \left(\bigoplus_{j=1}^{m'}n_{j}'\delta^{t_{j}'}\right) \odot s'' \preceq p;
\]
according to the distributivity of $\odot$ over $\oplus$, the latter inequality is equivalent to
\[
    \bigoplus_{j=1}^{m'} (n_{j}'\delta^{t_j'}\odot s'') \preceq p.
\]
From Remark~\ref{rem:oplus_property}, the inequality can be rewritten as the system
\[
   \left\{
       \begin{array}{l}
            n_1'\delta^{t_1'} \odot s'' \preceq p\\
            n_2'\delta^{t_2'} \odot s'' \preceq p\\
            \dots\\
            n_m'\delta^{t_m'} \odot s'' \preceq p
       \end{array}
   \right. ,
\]
which, from Remark~\ref{rem:residual_property}, is equivalent to
\[
   \left\{
       \begin{array}{l}
            s'' \preceq p\odot^\sharp n_1'\delta^{t_1'}\\
            s'' \preceq p\odot^\sharp n_2'\delta^{t_2'}\\
            \dots\\
            s'' \preceq p\odot^\sharp n_m'\delta^{t_m'} 
       \end{array}
   \right. .
\]
Using again Remark~\ref{rem:oplus_property}, we get
\[
    s'' \preceq \bigwedge_{j=1}^{m'} \left( p \odot^\sharp n_j'\delta^{t_j'}\right) = \bigwedge_{j=1}^{m'} \left(\left( \bigoplus_{i = 1}^m n_i\delta^{t_i}\right) \odot^\sharp n_j' \delta^{t_j'}\right);
\]
finally, thanks to Lemma~\ref{le:odot_sharp_poly_mono}, this last expression can be rewritten as the formula in the statement of the proposition, since from the definition of residual we can substitute relation "$\preceq$" with "$=$".
\end{proof}

As shown in the next example, because of the non-increasing property of counters, the computation $p\odot^\sharp p'$ can occasionally be simplified.
If $t_i\geq t_j'$ for a pair of monomials, then we can continue with $j+1$ and setting $i$ back to $1$.
The rest of the terms would namely all have $+\infty$ as their $\delta$-exponent, with (in the standard meaning) increasing coefficients.

\begin{example}
The residual $p \odot^\sharp p'$ with $p = 2\delta^1 \oplus 3\delta^5 \oplus 5\delta^7$, $p' = 0\delta^2 \oplus 2\delta^5 \oplus 4\delta^{+\infty}$ can be determined by
\[
   \begin{array}{rl}
       p\odot^\sharp p' =& (2\delta^1 \oplus 3\delta^{+\infty} \oplus \cancel{5\delta^{+\infty}}) \wedge (0\delta^1 \oplus 1\delta^{+\infty} \oplus \cancel{3\delta^{+\infty}}) \\
       &\wedge (-2\delta^1 \oplus -1\delta^5 \oplus 1\delta^7)\\
       =& 2\delta^1 \oplus 3\delta^7.
   \end{array}
\]
\end{example}

\subsubsection{The residual of the Hadamard product for periodic series}

Before considering operation $\odot^\sharp$ between two periodic series, it is convenient to prove the following lemma, which concerns the series whose coefficients are computed by subtracting (in the standard sense) element-wise coefficients of two counters.
We remark that the resulting series is not necessarily a counter.

\begin{proposition}[Difference series]\label{pr:difference_series}
Let $s,s'\in\Sigma$ be periodic series such that there is no $t\in\Z$ for which $s(t)=s'(t)\in\{-\infty,+\infty\}$.
The series $\bar{s}$ defined by $\bar{s}(t) = s(t) - s'(t)$ for all $t\in\Z$ is periodic, with periodic pattern starting at time $\bar{t}_\textup{p}$ and throughput $\frac{\bar{\nu}}{\bar{\tau}}$, where
\[
    \bar{\tau} = \textup{lcm}(\tau,\tau'),\quad \bar{\nu} = \bar{\tau}\left(\frac{\nu}{\tau}-\frac{\nu'}{\tau'}\right),\quad \bar{t}_{\textup{p}} = \max(T_1,T'_1).
\]
\end{proposition}
\begin{proof}
Let $\bar{\tau} = \mbox{lcm}(\tau,\tau')=\alpha\tau=\alpha'\tau'$.
Then, since counters $s$ and $s'$ are periodic, for all $t\geq\max(T_1,T_1')=\bar{t}_\textup{p}$, $k\geq 0$,
\[
    \begin{array}{rl}
        \bar{s}(t+k\bar{\tau}) &= s(t+k\bar{\tau}) - s'(t+k\bar{\tau}) = s(t+k\alpha\tau) - s'(t+k\alpha'\tau')\\
        &= k\alpha\nu + s(t) - k\alpha'\nu' - s'(t) = k(\alpha\nu-\alpha'\nu')+\bar{s}(t).
    \end{array}
\]
Therefore, according to Proposition~\ref{pr:periodic_series}, series $\bar{s}$ is periodic with $\bar{\tau}$, $\bar{\nu}$, and $\bar{t}_\textup{p}$ as in the statement of the proposition.
\end{proof}

\begin{proposition}[$\odot^\sharp$ for periodic series]
Let $\bar{s}$ be the difference series defined as in Proposition~\ref{pr:difference_series}, and let $\bar{\nu}>0$.
Series $s'' = s\odot^\sharp s'$ is periodic, with periodic pattern starting at the latest at time $t_\textup{p}'' = \bar{t}_\textup{p}+ \kappa \tau''$ and throughput $\frac{\nu''}{\tau''}$, where
\[
    \kappa=1+\max\left(0,\ceil{\frac{\displaystyle\left(\max_{i=\bar{t}_1}^{\bar{t}_\textup{p}-1}\bar{s}(i)\right)-\left(\max_{j=\bar{t}_\textup{p}}^{\bar{t}_\textup{p}+\tau''-1}\bar{s}(j)\right)}{\nu''}}\right),
\] 
\[
   \bar{t}_1 = \min(t_1,t_1'), \quad \tau'' = \textup{lcm}(\tau,\tau'), \quad \mbox{and}\quad \nu'' = \tau''\left(\frac{\nu}{\tau}-\frac{\nu'}{\tau'}\right).
\]
\end{proposition}
\begin{proof}
Series $s''$ is the greatest counter which is less than or equal to $\bar{s}$.
In other words, $s''$ is the greatest series that satisfies the following two specifications:
\[
    s'' \preceq \bar{s}, \quad\mbox{and}\quad s''(t)\preceq \bigwedge_{i\leq t}s''(i) \quad \forall t\in\Z.
\]
Combining the two properties, we get that $s''$ is the greatest series satisfying
\[
    s''(t) \preceq \bigwedge_{i\leq t} \bar{s}(i)\quad \forall t\in\Z;
\]
equivalently, we can define $s''$ as the series such that
\[
    s''(t) = \bigwedge_{i\leq t} \bar{s}(i) = \bigwedge_{i\leq t}s(i) - s'(i) \quad \forall t\in\Z.
\]
Let $\bar{t}_1 = \min(t_1,t'_1)$; since $s$ and $s'$ are constant for all $t\leq \bar{t}_1$, we have
\[
    s''(t) = \bigwedge_{i = \bar{t}_1}^{t} \bar{s}(i) \quad \forall t\in\Z.
\]

To show that $s''$ is periodic, we now look for the time at which the periodic pattern of $s''$ starts.
Let us start by considering a time $t = \bar{t}_\textup{p} + k \bar{\tau} - 1$, where $k > 0$; then,
\[
    s''(t) = \bigwedge_{i = \bar{t}_1}^{\bar{t}_\textup{p}-1} \bar{s}(i) \wedge \bigwedge_{i = \bar{t}_\textup{p}}^{\bar{t}_\textup{p}+\bar{\tau}-1} \bar{s}(i) \wedge \bigwedge_{i = \bar{t}_\textup{p}+\bar{\tau}}^{\bar{t}_\textup{p}+2\bar{\tau}-1} \bar{s}(i) \wedge \dots \wedge \bigwedge_{i = \bar{t}_\textup{p}+(k-1)\bar{\tau}}^{\bar{t}_\textup{p}+k\bar{\tau}-1} \bar{s}(i).
\]
Due to the periodicity of $\bar{s}$, the distributivity of $\otimes$ over finite $\wedge$ in $\Zminbar$ when none of the operands is $\pm\infty$, and the fact that $\bar{\nu}>0$, the latter expression can be rewritten as
\begin{align}\label{eq:s_second_periodic}
    s''(t) &= \bigwedge_{i = \bar{t}_1}^{\bar{t}_\textup{p}-1} \bar{s}(i) \wedge \cancel{\bigwedge_{i = \bar{t}_\textup{p}}^{\bar{t}_\textup{p}+\bar{\tau}-1} \bar{s}(i)} \wedge \cancel{\bar{\nu}\left(\bigwedge_{i = \bar{t}_\textup{p}}^{\bar{t}_\textup{p}+\bar{\tau}-1} \bar{s}(i)\right)} \wedge \dots \wedge \bar{\nu}^{k-1}\left(\bigwedge_{i = \bar{t}_\textup{p}}^{\bar{t}_\textup{p}+\bar{\tau}-1} \bar{s}(i)\right) \nonumber\\
    &= \bigwedge_{i = \bar{t}_1}^{\bar{t}_\textup{p}-1} \bar{s}(i) \wedge \bar{\nu}^{k-1}\left(\bigwedge_{i = \bar{t}_\textup{p}}^{\bar{t}_\textup{p}+\bar{\tau}-1} \bar{s}(i)\right).
\end{align}
Since $\bar{\nu}>0$, the second term dominates the first one for $k$ large enough.
In case the least integer $k$ for which this happens is positive, this value can be easily found from the expression above, and coincides with 
\begin{equation}\label{eq:kappa_positive}
    1+\ceil{\frac{\max_{i=\bar{t}_1}^{\bar{t}_\textup{p}-1}\bar{s}(i)-\max_{j=\bar{t}_\textup{p}}^{\bar{t}_\textup{p}+\bar{\tau}-1}\bar{s}(j)}{\bar{\nu}}},
\end{equation}
where $\ceil{x}$ is the least integer greater than or equal to $x$ (in the standard sense).
However, if the least integer $k$ for which the second term in~\eqref{eq:s_second_periodic} dominates the first one is non-positive, the formula above may produce a wrong result, as it has been derived under the assumption that $k>0$.
In any case, taking $k$ equal to the maximum between~\eqref{eq:kappa_positive} and $1$ guarantees that the second term in~\eqref{eq:s_second_periodic} dominates the first one.
Let us define $\kappa$ as such value of $k$; in standard algebra, we get
\[
    \kappa=1+\max\left(0,\ceil{\frac{\max_{i=\bar{t}_1}^{\bar{t}_\textup{p}-1}\bar{s}(i)-\max_{j=\bar{t}_\textup{p}}^{\bar{t}_\textup{p}+\bar{\tau}-1}\bar{s}(j)}{\bar{\nu}}}\right).
\]

We now have everything needed to show that $s''$ is periodic. 
Take a time $t$ such that $\bar{t}_\textup{p}+\kappa\bar{\tau} \leq t < \bar{t}_\textup{p}+(\kappa+1)\bar{\tau}$, say $t=\bar{t}_\textup{p}+\kappa\bar{\tau} + h$ with $h\in \{0,\ldots,\bar{\tau}-1\}$; using~\eqref{eq:s_second_periodic} and $k=\kappa$ with $\kappa$ as defined above, we get
\begin{align*}
    s''(t) &= \cancel{\bigwedge_{i = \bar{t}_1}^{\bar{t}_\textup{p}-1}s(t)} \wedge \bar{\nu}^{\kappa-1}\left(\bigwedge_{i = \bar{t}_\textup{p}}^{\bar{t}_{\textup{p}}+\bar{\tau}-1} \bar{s}(i) \right) \wedge \bar{\nu}^{\kappa}\left(\bigwedge_{i = \bar{t}_\textup{p}}^{\bar{t}_\textup{p}+h} s(i)\right)\\
    &= \bar{\nu}^{\kappa-1}\left(\bigwedge_{i = \bar{t}_\textup{p}}^{\bar{t}_{\textup{p}}+\bar{\tau}-1} \bar{s}(i) \right) \wedge \bar{\nu}^{\kappa}\left(\bigwedge_{i = \bar{t}_\textup{p}}^{\bar{t}_\textup{p}+h} s(i)\right),
\end{align*}
and, for all $k'\geq 0$,
\begin{align*}
    s''(t+k'\bar{\tau}) = \bar{\nu}^{\kappa+k'-1}\left(\bigwedge_{i = \bar{t}_\textup{p}}^{\bar{t}_{\textup{p}}+\bar{\tau}-1} \bar{s}(i) \right) \wedge \bar{\nu}^{\kappa+k'}\left(\bigwedge_{i = \bar{t}_\textup{p}}^{\bar{t}_\textup{p}+h} s(i)\right) = \bar{\nu}^{k'} s''(t),
\end{align*}
which, in standard algebra, reads $s''(t+k'\bar{\tau}) = k'\bar{\nu}+s''(t)$.
This proves that $s''$ is periodic, with periodic pattern starting at the latest at time $\bar{t}_\textup{p}+\kappa\bar{\tau}$ and throughput $\frac{\nu''}{\tau''}$, where $\tau'' = \bar{\tau}$ and $\nu'' = \bar{\nu}$.
\end{proof}

\begin{example}
Consider $s = -1\delta^0 \oplus (5\delta^2)(2\delta^1)^*$ and $s' = -1\delta^{-5} \oplus (3\delta^0)(1\delta^2)^*$ from $\Sigma$. 
We determine $\bar{t}_\textup{p} = 2,\, \bar{\tau} = \tau'' = 2$ and $\bar{\nu} = \nu'' = 3$. 
To compute $\kappa$, the series are expanded up to $\bar{t}_\textup{p} + \bar{\tau}-1 = 3$ first, yielding $-1\delta^0 \oplus 5\delta^2 \oplus 7\delta^3$ and $-1\delta^{-5} \oplus 3\delta^0 \oplus 4\delta^2 \oplus 5\delta^3$, respectively. 
With $\bar{t}_1 = -5$ and $\bar{t}_\textup{p} + \bar{\tau} - 1 = 3$, this leads to $\kappa = 1 + \ceil{\frac{1-2}{3}} = 1$. 
The periodic pattern of $s''$ therefore starts at $t''_\textup{p} = \bar{t}_\textup{p} + \kappa\tau'' = 4$ at the latest. 
The final result is $s'' = s\odot^\sharp s' = 0\delta^0 \oplus (1\delta^2 \oplus 2\delta^3)(3\delta^2)^*$. 
In this case, the periodic behavior of $s''$ in fact starts at time $2 < t_\textup{p}''$. 
\end{example}

If $\bar{\nu} < 0$, then the coefficients of $s'$ will, on average, increase faster than those of $s$ (in the standard sense), from time $\bar{t}_\textup{p}$. 
Being the least counter greater than or equal to $\bar{s}$, $s'' = s\odot^\sharp s'$ will therefore remain constant after time $\bar{t}_\textup{p} + \bar{\tau} - 1$. 
This means that the information contained in $s - s'$ up to $\bar{t}_\textup{p} + \bar{\tau} - 1$ is sufficient to determine $s''$. 
Consequently, it suffices to expand $s$ and $s'$ up to $\bar{t}_\textup{p} + \bar{\tau} - 1$, compute the Hadamard residual of the obtained polynomials, and set the last $\delta$-exponent of the result to $+\infty$. 
The same approach can be taken for $\bar{\nu} = 0$, since the coefficients of $s$ in this case will, on average, change by the same amount as those of $s'$. 

Special cases of this operation between a polynomial and a periodic series again have to be examined. 
Considering $s\odot^\sharp p'$, if $t'_m = +\infty$, then $p'$ can be written as $p' = (\bigoplus_{j=1}^{m'-1} n'_i\delta^{t'_i})\oplus n'_m\delta^{1+t_{m-1}'}(0\delta^1)^*$, similarly to the case for the Hadamard product between a series and a polynomial. 
The rest of the algorithm can then be applied as previously described. 
In the computation of $p\odot^\sharp s'$, if the case $t_m = +\infty$ occurs, then the result can also be described by a polynomial. 
The series $s'$ is expanded up to $1 + t_{m-1}$, since $p\odot^\sharp s'$ will be constant from this time on. 
The operation $\odot^\sharp$ is then performed, and the last $\delta$-exponent of the result can be set to $+\infty$, without having to determine the value of $\kappa$ (this approach is analogous to that in case $\bar{\nu} < 0$).

For $p\odot^\sharp s'$ (resp. $s\odot^\sharp p'$), if $t_m \neq +\infty$ (resp. $t_m' \neq +\infty$), the value of $\kappa$ does not have to be determined, either, since the result can also be given as a polynomial (where the last $\delta$-exponent is $+\infty$ for $s\odot^\sharp p'$). 
Expanding $s'$ (resp. $s$) up to $t_m+1$ (resp. $t_m'$) to obtain $\tilde{p}'$ (resp. $\tilde{p}$) is sufficient, and the result can then be obtained by $p\odot^\sharp \tilde{p}'$ (resp. $\tilde{p} \odot^\sharp p'$). 
Note that $s'$ needs to be expanded up to $t_m+1$ and not only to $t_m$, otherwise the $\delta$-exponent of the last monomial of the result would be $+\infty$ instead of $t_m$ (see the formula of $\odot^\sharp$ for monomials when $\tau = \tau'$).

\subsection{The dual residual of the Hadamard product}

\subsubsection{The dual residual of the Hadamard product for monomials}

\begin{proposition}[$\odot^\flat$ for monomials]
The following rule holds:
\[
    r\odot^\flat r' = 
    \begin{dcases}
        (\nu-\nu')\delta^{\tau} & \textup{if }\tau\leq\tau'\\
        \textup{undefined} & \textup{otherwise}
    \end{dcases},
\]
with the convention that $+\infty-\infty=+\infty$.
\end{proposition}
\begin{proof}
As follows, we only consider $\nu,\nu',\tau,\tau'\in\Z$; the cases in which $\nu,\nu',\tau,\tau'\in\{-\infty,+\infty\}$ have to be considered separately, but can be easily checked by inspection.
Counter $r\odot^\flat r'$ is the least solution of $r'\odot x\succeq r$; we distinguish two cases.

\textbf{Case 1: $\tau\leq\tau'$.}
Equivalently, $r\odot^\flat r'$ is the counter that solves the following multi-objective optimization problem (written in standard algebra for simplicity; consider always $+\infty-\infty=+\infty$)\footnote{Note that the existence and uniqueness of the solution to this and the following multi-objective optimization problem is guaranteed by the fact that $\Pi_a:\D_a\rightarrow\CCC_a$, $x \mapsto a\odot x$ is dually residuated for all $a\in\Sigma$, cf. Proposition~\ref{pr:dual_residual}.}:
\[
    \begin{array}{rll}
    \max        & x(t) &  \forall t\in\Z, \\
    \mbox{s.t.} & \nu' + x(t) \leq \nu  & \forall t \leq \tau,\\
                & \nu' + x(t) \leq +\infty & \forall \tau < t \leq \tau',\\
                & +\infty + x(t) \leq +\infty & \forall \tau' < t,\\
                & x(t) \leq x(t+1) & \forall t\in\Z.
    \end{array}
\]
It is then clear that the least counter that satisfies $r'\odot x\succeq r$ has coefficients
\[
    x(t) = 
    \begin{dcases}
        \nu - \nu' & \mbox{if } t \leq \tau,\\
        +\infty & \mbox{otherwise}
    \end{dcases},
\]
which corresponds to monomial $(\nu-\nu')\delta^{\tau}$.

\textbf{Case 2: $\tau>\tau'$.}
Note that, for all $\tau' \leq t < \tau$, $r'(t) = +\infty$ and $r(t)\in\Z$.
Thus, from Proposition~\ref{pr:dual_residual}, $\Pi_{r'}^\flat(r) = r \odot^\flat r'$ is not defined, as $r$ does not belong to its domain $\CCC_{r'} = \{y\in\Sigma\ | \ y(t) = +\infty\ \forall t\in\Z\ \mbox{such that}\ r'(t)\in\{-\infty,+\infty\}\}$.
Indeed, as discussed in Section~\ref{se:2}, the Hadamard product is not dually residuated in this case.
\end{proof}

\subsubsection{The dual residual of the Hadamard product for polynomials}

Before giving the rule for computing $\odot^\flat$ on polynomials, we need to consider the special case in which the left operand is a monomial and the right operand a polynomial. 

\begin{lemma}[$\odot^\flat$ between a monomial and a polynomial]\label{le:odot_flat_mono_poly}
The following rule holds:
\[
    r\odot^\flat p = \nu\delta^\tau \odot^\flat \left(\bigoplus_{i=1}^m n_i\delta^{t_i}\right)= 
    \begin{dcases}
        \bigwedge_{i=1}^m \left(\nu\delta^\tau\odot^\flat n_i\delta^{t_i}\right) & \textup{if } \tau \leq t_m\\
        \textup{undefined} & \textup{otherwise},
    \end{dcases}
\]
where undefined intermediate computations are to be ignored.
\end{lemma}
\begin{proof}
Let $j\in\{1,\ldots,m,m+1\}$ be the smallest index for which $t_j \geq \tau$, where $t_{m+1} = +\infty$.
We will assume that $\tau\leq t_m$; otherwise, the operation is clearly not defined -- the case in which $j = m+1$ will thus be ignored.
We also define $t_0 = -\infty$ for convenience.
Ignoring the intermediate computations that would lead to undefined results, i.e., $\nu\delta^\tau \odot^\flat n_i \delta^{t_i}$ for all $i < j$, we get
\begin{equation}\label{eq:mono_flat_poly}
    \bigwedge_{i=1}^m \left(\nu\delta^{\tau} \odot^\flat n_i\delta^{t_i}\right) = \bigwedge_{i=j}^m \left(\nu\delta^{\tau} \odot^\flat n_i\delta^{t_i}\right) = \bigwedge_{i=j}^m \left(\nu-n_i\right)\delta^\tau = (\nu-n_j)\delta^\tau.
\end{equation}
We need to prove that $r\odot^\flat p$ leads to the same expression.

From the definition of dual residual,
\[
    r\odot^\flat p = \bigwedge\{x\in\Sigma\ |\ p\odot x \succeq r\}.
\]
Equivalently, $r\odot^\flat p$ is the counter that solves the following multi-objective optimization problem (written in standard algebra for simplicity; consider always $+\infty-\infty=+\infty$):
\[
    \begin{array}{rll}
    \max        & x(t) &  \forall t\in\Z, \\
    \mbox{s.t.} & n_i + x(t) \leq \nu  & \forall i\in\{1,\ldots,j-1\},\ t_{i-1} < t \leq t_i,\\
                & n_j + x(t) \leq \nu & \forall t_{j-1} < t \leq \tau,\\
                & n_j + x(t) \leq +\infty & \forall \tau < t \leq t_{j},\\
                & n_i + x(t) \leq +\infty & \forall i\in\{j+1,\ldots,m\},\ t_{i-1} < t \leq t_{i},\\
                & +\infty + x(t) \leq +\infty & \forall t_m < t,\\
                & x(t) \leq x(t+1) & \forall t\in\Z.
    \end{array}
\]
Clearly, the optimization problem can be rewritten as 
\[
    \begin{array}{rll}
    \max        & x(t) &  \forall t\in\Z, \\
    \mbox{s.t.} & x(t) \leq \nu - n_i & \forall i\in\{1,\ldots,j-1\},\ t_{i-1} < t \leq t_i,\\
                & x(t) \leq \nu - n_j & \forall t_{j-1} < t \leq \tau,\\
                & x(t) \leq +\infty & \forall \tau < t,\\
                & x(t) \leq x(t+1) & \forall t\in\Z,
    \end{array}
\]
whose solution is, as $\nu-n_i<\nu-n_{i+1}$ for all $i$,
\[
    x(t) = 
    \begin{dcases}
        \nu - n_j & \forall t \leq \tau,\\
        +\infty & \forall \tau < t;
    \end{dcases}
\]
written using the $\delta$-transform, the above expression coincides with~\eqref{eq:mono_flat_poly}.
\end{proof}

\begin{proposition}[$\odot^\flat$ for polynomials]
The following rule holds:
\[
    p\odot^\flat p' = 
    \begin{dcases}
    \bigoplus_{i=1}^{m}\bigwedge_{j=1}^{m'} n_i\delta^{t_i}\odot^\flat n_j'\delta^{t_j'} & \textup{if }t_m \leq t_m',\\
    \textup{undefined} & \textup{otherwise},
    \end{dcases}
\]
where undefined intermediate computations are to be ignored.
\end{proposition}
\begin{proof}
We skip the case $t_m>t_m'$, as it obviously leads to an undefined result.
Let $s''=p\odot^\flat p'$ be the least counter satisfying
\[
    p'\odot s'' \succeq p, \quad i.e., \quad p' \odot s'' \succeq \bigoplus_{i=1}^m n_i\delta^{t_i},
\]
From Remark~\ref{rem:oplus_property}, the inequality can be rewritten as the system
\[
   \left\{
       \begin{array}{l}
            p' \odot s'' \succeq n_1\delta^{t_1}\\
            p' \odot s'' \succeq n_2\delta^{t_2}\\
            \dots\\
            p' \odot s'' \succeq n_m\delta^{t_m}
       \end{array}
   \right. ,
\]
which, from Remark~\ref{rem:residual_property}, is equivalent to
\[
   \left\{
       \begin{array}{l}
            s'' \succeq n_1\delta^{t_1}\odot^\flat p'\\
            s'' \succeq n_2\delta^{t_2}\odot^\flat p'\\
            \dots\\
            s'' \succeq n_m\delta^{t_m}\odot^\flat p' 
       \end{array}
   \right. .
\]
Using again Remark~\ref{rem:oplus_property}, we get
\[
    s'' \succeq \bigoplus_{i=1}^{m} \left( n_i\delta^{t_i} \odot^\flat p'\right) = \bigoplus_{i=1}^{m} \left(n_i \delta^{t_i} \odot^\flat \left( \bigoplus_{j = 1}^{m'} n'_j\delta^{t'_j}\right)\right);
\]
finally, thanks to Lemma~\ref{le:odot_flat_mono_poly}, this last expression can be rewritten as the formula in the statement of the proposition, since from the definition of dual residual we can substitute relation "$\succeq$" with "$=$".
\end{proof}

Observe that, when performing $\odot^\flat$ on polynomials, if for a pair $i,j$ the computation $n_i\delta^{t_i} \odot^\flat n_j'\delta^{t_j'}$ is defined, we can proceed with $i+1$ and $j=1$ for the next term.
Indeed, when $j$ is increased, the difference $n_i-n_j'$ will be decreasing (in the standard sense) as $p'$ is a counter and $i$ does not change.
Because of this, the remaining terms will not influence the result of the greatest lower bound, and can therefore be neglected.

\begin{example}
Let $p = -3\delta^{-1}\oplus -1\delta^{2}\oplus 3\delta^{4}$ and $p' = -1\delta^{1}\oplus e\delta^{3}\oplus 2\delta^{4}$. 
Then,
\begin{align*}
    p\odot^\flat p' &= (\underline{-2\delta^{-1}}\wedge -3\delta^{-1}\wedge -5\delta^{-1})\oplus (\underline{-1\delta^{2}}\wedge -3\delta^{2})\oplus (\underline{1\delta^{4}})\\
    &= -2\delta^{-1}\oplus -1\delta^{2}\oplus 1\delta^{4}.
\end{align*}
It can also be seen from this example that in each iteration of the greatest lower bound, there is no need to consider terms after we find a term which is defined, since the first term will always be the least one in the current iteration.
\end{example}

\subsubsection{The dual residual of the Hadamard product for periodic series}

\begin{proposition}[$\odot^\flat$ for periodic series]
Let $\bar{s}$ be the difference series defined as in Proposition~\ref{pr:difference_series}, and let $\bar{\nu}>0$.
Series $s'' = s\odot^\flat s'$ is periodic, with periodic pattern described by
\[
    \tau'' = \textup{lcm}(\tau,\tau'), \quad \nu'' = \tau''\left(\frac{\nu}{\tau}-\frac{\nu'}{\tau'}\right), \quad t''_\textup{p} = \max(T_1,T_1').
\]
\end{proposition}
\begin{proof}
Series $s''$ is the least counter which is greater than or equal to $\bar{s}$.
In other words, $s''$ is the least series that satisfies the following two specifications:
\[
    s\succeq \bar{s}, \quad\mbox{and}\quad s(t)\succeq \bigoplus_{i\geq t}s(i) \quad \forall t\in\Z.
\]
Combining the two properties, we get that $s''$ is the least series satisfying
\[
    s''(t) \succeq \bigoplus_{i\geq t} \bar{s}(i)\quad \forall t\in\Z;
\]
equivalently, we can define $s''$ as the series such that
\[
    s''(t) = \bigoplus_{i\geq t} \bar{s}(i) = \bigoplus_{i\geq t}s(i) - s'(i) \quad \forall t\in\Z.
\]
Due to the periodicity of $\bar{s}$ and the fact that $\bar{\nu}>0$, for all $t\geq \bar{t}_\textup{p}$ we can write
\begin{align*}
    s''(t) &= \bigoplus_{i = t}^{t+\bar{\tau}-1} \bar{s}(i) \oplus \bigoplus_{i = t+\bar{\tau}}^{t+2\bar{\tau}-1} \bar{s}(i) \oplus \bigoplus_{i = t+2\bar{\tau}}^{t+3\bar{\tau}-1} \bar{s}(i) \oplus \ldots\\
    &= \bigoplus_{i = t}^{t+\bar{\tau}-1} \bar{s}(i) \oplus \cancel{\nu\left(\bigoplus_{i = t}^{t+\bar{\tau}-1} \bar{s}(i)\right)} \oplus \cancel{\nu^2 \left(\bigoplus_{i = t}^{t+\bar{\tau}-1} \bar{s}(i)\right)} \oplus \ldots \\
    &= \bigoplus_{i = t}^{t+\bar{\tau}-1} \bar{s}(i).
\end{align*}
Moreover, taking again $t\geq \bar{t}_\textup{p}$ and any integer $k\geq 0$, we also get
\begin{align*}
    s''(t+k\bar{\tau}) &= \bigoplus_{i = t + k\bar{\tau}}^{t + (k+1)\bar{\tau}-1} \bar{s}(i) = \bar{\nu}^{k} \left(\bigoplus_{i = t}^{t+\bar{\tau}-1} \bar{s}(i)\right) = \bar{\nu}^{k} s''(t),
\end{align*}
which, in standard algebra, reads $s''(t+k\bar{\tau}) = k\bar{\nu}+s''(t)$.
This proves that $s''$ is periodic, with periodic pattern starting at the latest at time $t''_\textup{p} = \bar{t}_\textup{p}$ and throughput $\frac{\nu''}{\tau''}$, where $\tau'' = \bar{\tau}$ and $\nu'' = \bar{\nu}$.
\end{proof}


As with the residual of the Hadamard product, the case $\bar{\nu} < 0$ can occur when computing the dual residual as well. 
Since in this case the values of $\bar{s}$ keep decreasing (in standard algebra) with each period, the only counter which is greater than or equal to $\bar{s}$ is $s_\top$. 
If $\bar{\nu} = 0$, then the result will be a polynomial whose last $\delta$-exponent is $+\infty$. 

\begin{example}
Let $s = 2\delta^2 \oplus 6\delta^3 (6\delta^8)^*$ and $s' = 1\delta^1 \oplus 5\delta^4 (3\delta^4)^*$. 
With $t''_{\textup{p}} = \max(3,4) = 4,\, \tau'' = \mbox{lcm}(8,4) = 8$ and $\nu'' = 0$, expanding up to $t''_\textup{p} + \tau'' - 1 = 11$ yields
\[ 
    \tilde{p} = 2\delta^2 \oplus 6\delta^3 \oplus 12\delta^{11} ,\quad \tilde{p}' = 1\delta^1 \oplus 5\delta^4 \oplus 8\delta^8 \oplus 11\delta^{11}.
\]
We obtain $\tilde{p} \odot^\flat\tilde{p}' = -3\delta^2 \oplus 1\delta^3 \oplus 1\delta^{11} = -3\delta^2 \oplus 1\delta^{11}$, so that the final result equals $s\odot^\flat s' = -3\delta^2 \oplus 1\delta^{11}(0\delta^8)^* = -3\delta^2 \oplus 1\delta^{+\infty}$.
\end{example}

The special case where one of the operands is a polynomial and the other is a periodic series needs to be examined as well. 
First of all, $p \odot^\flat s'$ results in $s_\top$ if $t_m = +\infty$, since $p$ in this case can also be represented as a periodic series with $\nu = 0$, which leads to $\nu'' < 0$. 
This case has been discussed above. 
If $t_m \neq +\infty$, then the result can also be represented by a polynomial, i.e., $\nu'' = 0$ will hold. 
This result is given by $p\odot^\flat\tilde{p}'$, where $\tilde{p}'$ is a polynomial obtained from expanding $s'$ up to $t_m$.

In the alternative case of $s \odot^\flat p'$, if $t_m' \neq +\infty$, the operation will be undefined, since there will be $+\infty$-coefficients in $p'$, of which the corresponding coefficients in $s$ are not $+\infty$ (unless $s = s_\varepsilon$ holds). 
If $t_m = +\infty$, then (similarly to the corresponding case for the Hadamard product), the series $s$ can be expanded up to $t''_\textup{p} + \tau'' - 1$, with $t''_\textup{p} = \max(T_1, 1 + t_{m-1}')$. 
The result can then be determined in the same way as previously described, and its periodicity will be given by the monomial $r'' = \nu \delta^{\tau}$.

\section{Software implementation tutorial}\label{se:4}

In this section, we show how to call the \CC~functions for computing the Hadamard product and its residuals using the library ETVO.

\subsection{Operations with monomials}

Consider monomials $r_1 = 5\delta^2$ and $r_2 = 3\delta^2$.
The following code produces the result of 
\[
   r_1\odot r_2 = 8\delta^2, \quad  r_1\odot^\sharp r_2 = 2\delta^{+\infty}, \quad \mbox{and}\quad r_1\odot^\flat r_2 = 2\delta^2.
\]

\begin{lstlisting}[label={lst:monomials},caption=Operations for monomials]
void main()
{
    // initialize monomials
    gd r_1, r_2, res_odot, res_odot_sharp, res_odot_flat;

    // define $r_1$ and $r_2$
    r_1 = gd(5, 2); 
    r_2 = gd(3, 2);

    // print $r_1$ and $r_2$
    cout << "r_1 = " << r_1 << endl;
    cout << "r_2 = " << r_2 << endl;

    // compute $r_1\odot r_2$, $r_1\odot^\sharp r_2$, and $r_1\odot^\flat r_2$
    res_odot = hadamard_prod(r_1, r_2);
    res_odot_sharp = hadamard_res(r_1, r_2);
    res_odot_flat = hadamard_dualres(r_1, r_2);

    // print the results
    cout << "res_odot = " << res_odot << endl;
    cout << "res_odot_sharp = " << res_odot_sharp << endl;
    cout << "res_odot_flat = " << res_odot_flat << endl;
}
\end{lstlisting}

\subsection{Operations with polynomials}

Consider polynomials $p_1 = -3\delta^{-1} \oplus -1\delta^2 \oplus 3\delta^4$ and $p_2 = -1\delta^1 \oplus 0 \delta^3 \oplus 2\delta^4$.
The following code produces the result of
\[
    p_1 \odot p_2 = -4\delta^{-1} \oplus -2\delta^1 \oplus -1\delta^2 \oplus 3\delta^3 \oplus 5\delta^4
\]
\[
    p_1 \odot^\sharp p_2 = -2\delta^{-1} \oplus 0\delta^2 \oplus 3\delta^{+\infty},\quad
    \mbox{and} \quad p_1 \odot^\flat p_2 = -2\delta^{-1} \oplus -1\delta^2 \oplus 1\delta^4.
\]

\begin{lstlisting}[label={lst:polynomials},caption=Operations for polynomials]
void main()
{
    // initialize polynomials
    poly p_1, p_2, res_odot, res_odot_sharp, res_odot_flat;

    // define $p_1$ and $p_2$
    p_1 = gd(-3, -1) + gd(-1, 2) + gd(3, 4); 
    p_2 = gd(-1, 1) + gd(0, 3) + gd(2, 4);

    // print $p_1$ and $p_2$
    cout << "p_1 = " << p_1 << endl;
    cout << "p_2 = " << p_2 << endl;

    // compute $p_1\odot p_2$, $p_1\odot^\sharp p_2$, and $p_1\odot^\flat p_2$
    res_odot = hadamard_prod(p_1, p_2);
    res_odot_sharp = hadamard_res(p_1, p_2);
    res_odot_flat = hadamard_dualres(p_1, p_2);

    // print the results
    cout << "res_odot = " << res_odot << endl;
    cout << "res_odot_sharp = " << res_odot_sharp << endl;
    cout << "res_odot_flat = " << res_odot_flat << endl;
}
\end{lstlisting} 

\subsection{Operations with periodic series}

Consider series $s_1 = -1\delta^{0} \oplus (5\delta^2) (2\delta^1)^*$ and $s_2 = -1\delta^{-5} \oplus (3\delta^0) (1\delta^2)^*$.
The following code produces the result of
\[
    s_1 \odot s_2 = -2\delta^{-5} \oplus 2\delta^{0} \oplus (9\delta^2 \oplus 12\delta^3)(5\delta^2)^*
\]
\[
    s_1 \odot^\sharp s_2 = 0\delta^0 \oplus (1\delta^2 \oplus 2\delta^3)(3\delta^2)^*,\quad
    \mbox{and} \quad s_1 \odot^\flat s_2 = -4\delta^0 \oplus (1\delta^2 \oplus 2\delta^3)(3\delta^2)^*.
\]

\begin{lstlisting}[label={lst:series},caption=Operations for periodic series]
void main()
{
    // initialize series
    series s_1, s_2, res_odot, res_odot_sharp, res_odot_flat;

    // define $s_1$ and $s_2$
    s_1 = series(gd(-1, 0), gd(5, 2), gd(2, 1));
    s_2 = series(gd(-1, -5), gd(3, 0), gd(1, 2));

    // print $s_1$ and $s_2$
    cout << "s_1 = " << s_1 << endl;
    cout << "s_2 = " << s_2 << endl;

    // compute $s_1\odot s_2$, $s_1\odot^\sharp s_2$, and $s_1\odot^\flat s_2$
    res_odot = hadamard_prod(s_1, s_2);
    res_odot_sharp = hadamard_res(s_1, s_2);
    res_odot_flat = hadamard_dualres(s_1, s_2);

    // print the results
    cout << "res_odot = " << res_odot << endl;
    cout << "res_odot_sharp = " << res_odot_sharp << endl;
    cout << "res_odot_flat = " << res_odot_flat << endl;
}
\end{lstlisting}

\bibliographystyle{plainnat}
\bibliography{references}

\begin{thebibliography}{7}
\providecommand{\natexlab}[1]{#1}
\providecommand{\url}[1]{\texttt{#1}}
\expandafter\ifx\csname urlstyle\endcsname\relax
  \providecommand{\doi}[1]{doi: #1}\else
  \providecommand{\doi}{doi: \begingroup \urlstyle{rm}\Url}\fi

\bibitem[Baccelli et~al.(1992)Baccelli, Cohen, Olsder, and
  Quadrat]{baccelli1992synchronization}
Fran{\c{c}}ois Baccelli, Guy Cohen, Geert~Jan Olsder, and Jean-Pierre Quadrat.
\newblock \emph{{Synchronization and linearity: an algebra for discrete event
  systems}}.
\newblock John Wiley \& Sons Ltd, 1992.

\bibitem[Cohen(1993)]{cohen1993two}
Guy Cohen.
\newblock Two-dimensional domain representation of timed event graphs.
\newblock \emph{Summer School on DES}, 1993.

\bibitem[Cottenceau et~al.(2020)Cottenceau, Hardouin, and
  Trunk]{cottenceau2020c++}
Bertrand Cottenceau, Laurent Hardouin, and Johannes Trunk.
\newblock A {\CC} toolbox to handle series for event-variant/time-variant
  (max,+) systems.
\newblock 2020.

\bibitem[Hardouin et~al.(2008)Hardouin, Cottenceau, Lagrange, and
  Le~Corronc]{4605920}
Laurent Hardouin, Bertrand Cottenceau, Sebastien Lagrange, and Euriell
  Le~Corronc.
\newblock Performance analysis of linear systems over semiring with additive
  inputs.
\newblock In \emph{2008 9th International Workshop on Discrete Event Systems},
  pages 43--48, 2008.
\newblock \doi{10.1109/WODES.2008.4605920}.

\bibitem[Hardouin et~al.(2013)Hardouin, Cottenceau, and
  Lhommeau]{hardouin2013minmaxgd}
Laurent Hardouin, Bertrand Cottenceau, and Mehdi Lhommeau.
\newblock {MinMaxgd}, a toolbox to handle periodic series in semiring
  $\minmaxgd$.
\newblock \emph{University of Angers, France}, 2013.

\bibitem[Zorzenon et~al.(2022{\natexlab{a}})Zorzenon, Schafaschek, Tirpák,
  Moradi, Hardouin, and Raisch]{Hadamardproductcounters}
Davide Zorzenon, Germano Schafaschek, Dominik Tirpák, Soraia Moradi, Laurent
  Hardouin, and J\"org Raisch.
\newblock {H}adamard product for counters, 2022{\natexlab{a}}.
\newblock URL
  \url{https://github.com/davidezorzenon/Hadamard_product_counters}.

\bibitem[Zorzenon et~al.(2022{\natexlab{b}})Zorzenon, Schafaschek, Tirpák,
  Moradi, Hardouin, and Raisch]{zorzenon2022implementation}
Davide Zorzenon, Germano Schafaschek, Dominik Tirpák, Soraia Moradi, Laurent
  Hardouin, and J\"org Raisch.
\newblock Implementation of procedures for optimal control of timed event
  graphs with resource sharing.
\newblock 2022{\natexlab{b}}.
\newblock Sumbitted to the 16th IFAC Workshop on Discrete Event Systems.

\end{thebibliography}

\end{document}